\documentclass[]{interact}

\usepackage{epstopdf}
\usepackage[caption=false]{subfig}
\usepackage{amsfonts,color,soul,physics,amsmath,amssymb,amsthm,enumitem}
\usepackage{mathrsfs}
\usepackage[colorlinks=true, linkcolor=blue, anchorcolor=blue, citecolor=blue, filecolor=blue, urlcolor=blue]{hyperref}
\usepackage{graphicx}
\usepackage{float}

\usepackage[numbers,sort&compress]{natbib}
\bibpunct[, ]{[}{]}{,}{n}{,}{,}
\makeatletter
\def\NAT@def@citea{\def\@citea{\NAT@separator}}
\makeatother
\newtheorem{theorem}{Theorem}[section]
\newtheorem{prop}[theorem]{Proposition}

\newtheorem{lemma}[theorem]{Lemma}

\theoremstyle{definition}
\newtheorem{defn}[theorem]{Definition}
\newtheorem{eg}{Example}[section]
\theoremstyle{remark}
\newtheorem{rem}[theorem]{Remark}

\def\NN{\mathbb N}
\def\ZZ{\mathbb Z}
\def\RR{\mathbb R}

\def\oo{\mathcal O}
\def\aa{\mathcal A}
\def\ee{\mathcal E}

\def\suml{\sum\limits}
\def\supl{\sup\limits}

\begin{document}


\title{Random sampling and reconstruction in reproducing kernel subspace of mixed Lebesgue spaces}

\author{
\name{Prashant Goyal\textsuperscript{a}\thanks{Goyal Prashant. Email: goyalprashant194@gmail.com; Patel Dhiraj. Email: dpatel.iitd@gmail.com; Sampath Sivananthan. Email: siva@maths.iitd.ac.in}  Dhiraj Patel\textsuperscript{a} Sivananthan Sampath\textsuperscript{a}}
\affil{\textsuperscript{a}Department of Mathematics, Indian Institute of Technology Delhi, New Delhi-110016, India }
}

\maketitle

\begin{abstract}
In this article, we consider the random sampling in the image space $V$ of mixed Lebesgue space $L^{p,q}(\RR^{n+1})$ under an idempotent integral operator. We assume some decay and regularity conditions of the kernel and approximate the unit sphere in $V$ on a bounded cube $C_{R,S}$ by a finite-dimensional subspace of $V$. Consequently, the set of concentrated functions is totally bounded. We prove with an overwhelming probability that the random sample set uniformly distributed over $C_{R,S}$ is a stable set of sampling for the set of concentrated functions on $C_{R,S}$. Moreover, we propose an iterative scheme to reconstruct the concentrated signal from its random measurements.
\end{abstract}

\begin{keywords}
Reproducing kernel space; Mixed Lebesgue space; Random sampling; Reconstruction algorithm
\end{keywords}
\section{Introduction}
A \textit{signal} is an electrical current that carries data or information from one system to another. An \textit{analog signal} is a continuous time-varying function that works with physical values and natural phenomena such as temperature, sound, lighting, earth-quack, speed of wind, etc. A \textit{digital signal} is a discrete sequence that represents values of data at any point in time. In modern days, analog signals are analyzed and manipulated digitally. Therefore, digital data processing of signals uses the discretized version of the analog signal. This raises the question of whether we can recover the original signal from its discrete samples. The fundamental problem in sampling theory is to reconstruct or approximate a function $f$ uniquely and stably from its sampled values $\{ f(x_j):j\in J \}$ on some discrete set $X=\{x_j\}\subseteq \RR^n.$ Sampling theory is the most basic tool in signal analysis and modern pulse code modulation communication. It is one of the most active research due to its influential applications in digital signal processing, image processing, and wireless communication. As the characteristic properties of the signals are different for distinct signals, the sampling problem cannot be solved in general. In this article, we consider the function space as an image of an idempotent integral operator in mixed Lebesgue space $L^{p,q}(\RR^{n+1})$.

Mixed Lebesgue space $L^{p,q}(\RR^{n+1})$ is a generalization of Lebesgue space $L^{p}(\RR^{n+1})$ which contains measurable functions with independent variables having different properties. In particular, time-varying functions belong to mixed Lebesgue space. The mixed Lebesgue space $L^{p,q}(\RR^{n+1})$ is the collection of all measurable functions $f=f(x,y)$ such that
\begin{align*}
\|f\|_{L^{p,q}(\RR^{n+1})}^{q}=\int_{\RR} \Big(\int_{\RR^n}|f(x,y)|^{p}\,dx \Big)^{\frac{q}{p}}\,dy < \infty, ~~~\text{for } 1\leq p,q< \infty.
\end{align*}
The initial realization of mixed Lebesgue space traced back to the year 1960 of an interesting article by H\"{o}rmander \cite{HL} and first introduced by Benrdek and Panzone \cite{BP}. They considered mixed Lebesgue space $L^{\vec{p}}(X)$ of multi-variable measurable functions and norm on $L^{\vec{p}}(X)$ defined as iterative $L^p$ norm of the function with each variables. Mixed Lebesgue space $L^{p,q}(\RR^{n+1})$ inherits many properties of the standard $L^p(\RR^{n+1})$, such as the Plancherel-Polya inequality \cite{TW}, Leibniz-type rule \cite{HT}, and some classical conjectures in harmonic analysis \cite{CC}. The usefulness of mixed Lebesgue spaces motivates to study numerous other function spaces like Besov spaces, Sobolev spaces, and Bessel potential spaces with mixed norms \cite{BI79}. The classical spaces such as parabolic function space, Lorentz space, and Orlicz space were studied with mixed norms \cite{F77,G,M}. Hart et al. \cite{HT} introduced the mixed-norm Hardy space $H^{p,q}(\RR^{n+1})$ to improve the regularity of the bilinear operator. Moreover, due to the flexibility of independent variables of the function domain, interesting results were studied for the functions in mixed Lebesgue space in the context of partial differential equations. For example, to study the solution of partial differential equations involving both time and space variables, such as the heat or the wave equations. Functions in mixed Lebesgue space give precise information on the estimation of the parameters and induce better regularity for the solution of linear or non-linear equations \cite{B10}. We refer to \cite{HY} for more details on mixed Lebesgue space.\par

The sampling problem draws attention towards higher dimensional time-varying signals. In place of recording the value of signals $f(t)$ at a present time instant, one can also record the time at which the signal takes on the present value. This time-sampling approach is more practical in practice, such as in neuron models. The mixed Lebesgue spaces $L^{p,q}(\RR^{n+1})$ is ideal for modeling and measuring signals living in the time-space domain. The non-uniform sampling problem with mixed norm studied for shift-invariant spaces \cite{LLLZ,ZKD}, and image space of an idempotent integral operator \cite{JS,KPS}. \par 

Let $V$ be a reproducing kernel subspace of $L^{p,q}(\RR^{n+1})$. A set $X=\{ (x_i,y_j):i\in \ZZ^n, j\in \ZZ \}$ is a \textit{stable sampling set} for $V$ if there exist $A,B>0$ such that $$A\|f\|_{L^{p,q}(\RR^{n+1})}\leq \Big( \sum_{j \in \ZZ} \Big( \sum_{i \in \ZZ^n} |f(x_i,y_j)|^p \Big)^{\frac{q}{p}} \Big)^{\frac{1}{q}}\leq B\|f\|_{L^{p,q}(\RR^{n+1})} ~~\forall \, f\in V.$$ This implies that the sampling set should be dense enough so that $L^{p,q}$ norm is comparable to its discrete version. The Beurling density $D^{-}(X)$ of the sampling set $X$ characterizes the stability of the sampling set $X$ for Paley-Wiener space $PW_{[a,b]}(\RR).$ In particular, if $D^{-}(X)>1$, then the sampling set $X$ is a stable set of sampling for $PW_{[a,b]}(\RR).$ Similar stable sampling condition is not true for $PW_S(\RR^n)$, where $S$ is a convex subset of $\RR^n$. The zero set of a function in $PW_S(\RR^n)$ is an analytic manifold. Therefore, distributing points in the sampling set $X$ is quite challenging. These drawbacks drove to study the sampling problem in the context of probabilistic framework. The recovery of signals from random measurement is one of the well-known methods used in compressed sensing \cite{CRT,E} and machine learning \cite{CZ, P}. Therefore, the random sampling theory is one of the most popular areas of research nowadays. In recent years, the random sampling problem has been studied for multi-variable trigonometric polynomials \cite{BG}, band-limited functions \cite{BG10, BG13}, signals in a shift-invariant space \cite{FX,Y,YW13} and reproducing kernel subspace of $L^{p}(\RR^n)$ \cite{LSX, PS}. The mixed Lebesgue space $L^{p,q}(\RR^{n+1})$ inherits many properties of the classical $L^p(\RR^{n+1})$ space. However, studying the random sampling problem in mixed Lebesgue space $L^{p,q}(\RR^{n+1})$ brings various challenges due to the non-commutativity of the order of the integral. Recently, the random sampling problem studied for the shift-invariant subspace of mixed Lebesgue space $L^{p,q}(\RR^{n+1})$ \cite{JL}. In this article, we will discuss the random sampling problem for the image space of an idempotent integral operator in $L^{p,q}(\RR^{n+1})$. \par

In random sampling, sample points are randomly taken according to the probability distribution. The stability of the random sample set over unbounded set $\RR^{n+1}$ is implausible in general. Bass and Gr\"{o}chenig solved this problem by introducing the notion of relevant sampling \cite{BG10,BG13}. In this paper, we consider uniform probability distribution of the sampling set over the compact set $C_{R,S}=\Big[-\frac{R}{2},\frac{R}{2}\Big]^{n}\times \Big[-\frac{S}{2},\frac{S}{2}\Big]$ and prove the random sampling result for the class of functions concentrated on $C_{R,S}.$

The main result of this paper is to find the condition on the random sample set $X=\{(x_i,y_j):x_i\in \RR^n, y_j\in \RR, 1\leq i\leq l,1\leq j\leq m\}$ such that the following sampling inequality hold:
$$A\|f\|_{L^{p,q}(\RR^{n+1})}\leq \Big( \sum_{j=1}^m \Big( \sum_{i=1}^l |f(x_i,y_j)|^p \Big)^{\frac{q}{p}} \Big)^{\frac{1}{q}}\leq B\|f\|_{L^{p,q}(\RR^{n+1})}.$$ In other words, find the probability for which the concentrated function can be recovered from its random samples. Moreover, we provide an iterative reconstruction algorithm for the function $f$ concentrated on $C_{R,S}$. \par 

This article is divided into four sections. Section $2$ includes the basic definitions, notations, and preliminary results. In section $3$, we show the finite-dimensional approximation of the reproducing kernel space and prove that the set of concentrated functions on $C_{R,S}$ are totally bounded. The last section is devoted to studying probabilistic results on random sampling connecting to the main theorem. We also propose an iterative scheme to reconstruct the concentrated function from random samples.

\section{Preliminaries}

In this section, we define reproducing kernel space $V$. We assume some condition on integral kernel and discuss some preliminary properties of $V$.
\begin{defn}
For $1 \leq p, q < \infty,$ the mixed sequence space $\ell^{p,q}(\mathbb{Z}^{n+1})$ is a collection of all sequences $\{ c=c(i,j): i\in \mathbb{Z}^n, j\in \mathbb{Z} \}$ such that
\begin{align*}
\|c\|_{\ell^{p,q}}=\Big(\sum_{j\in \ZZ} \Big(\sum_{i\in \ZZ^n}|c(i,j)|^{p}  \Big)^{\frac{q}{p}}\Big)^{\frac{1}{q}} < \infty
\end{align*}
It is easy to verify that $\ell^{\infty,\infty}(\ZZ^{n+1})=\ell^{\infty}(\ZZ^{n+1})$ and $L^{\infty,\infty}(\RR^{n+1})=L^{\infty}(\RR^{n+1})$.
\end{defn}
For $1 \leq p,q < \infty$, we define the idempotent integral operator $T$ on $L^{p,q}(\RR^{n+1})$ with $T^2=T$ as 
\begin{align*}
Tf(x,y) = \int_{\RR}\int_{\RR^{n}} K(x,y,s,t)f(s,t)\,dsdt,
\end{align*}
where the symmetric kernel $K$ on $(\RR^n\times\RR)\times(\RR^n\times\RR)$ have decay with
\begin{align} \label{kernel2}
|K(x,y,s,t)| \leq \frac{c}{(1 + \|x-s\|_{1})^{\alpha}(1+|y-t|)^{\beta}},
\end{align}
where $c >0$, $\alpha >\frac{n}{p'}+n+2+\frac{1}{q'}$, $\beta > \frac{1}{q'}+n+2+ \frac{n}{q'}$ with $\frac{1}{p}+\frac{1}{p'} = 1$, $\frac{1}{q}+\frac{1}{q'}=1$ and 
$\|x\|_1 = \sum_{i=1}^{n} |x(i)|$, $x= (x(1),x(2)\dots x(n))\in\RR^n.$

In addition, assume that the kernel $K$ satisfy the regularity condition  
\begin{align} \label{kernel0}
\lim_{\epsilon \rightarrow 0}\|w_{\epsilon}(K)\|_{W}=\lim_{\epsilon \rightarrow 0}\Big\| \big\|w_{\epsilon}(K)\big\|_{W^{0}_{x,s}} \Big\|_{W^{0}_{y,t}}=0,
\end{align}
where $w_{\epsilon}(K)(x,y,s,t) = \supl_{\substack{\|(x',y')\|_{\infty}<\epsilon,\\ \|(s',t')\|_{\infty}<\epsilon}} |K(x+x',y+y',s+s',t+t') - K(x,y,s,t)|,$
and for the kernel $K_1$ define on $\RR^n\times\RR^n$, we denote $\|K_1\|_{W^0}$ as $$\|K_1\|_{W^{0}} = \max \Big\{\sup_{x\in\RR^n}\|K_1(x,\cdot)\|_{L^1(\RR^n)}, \sup_{y \in \RR^n}\|K_1(\cdot,y)\|_{L^1(\RR^n)} \Big\}.$$
Then the range space $V=\big\{ Tf: f\in L^{p,q}(\RR^{n+1}) \big\} = \big\{ f\in L^{p,q}(\RR^{n+1}): Tf=f \big\}$ of $T$ is a closed reproducing kernel subspace of $L^{p,q}(\RR^{n+1}).$ For $0<\delta<1$, the set of \textit{$\delta$-concentrated function} on $C_{R,S}$ is defined as follow:
\begin{align*}
V^{\star}(R,S,\delta)= \Big\{ f \in V : (1-\delta)\|f\|_{L^{p,q}(\RR^{n+1})}  \leq \|f\|_{L^{p,q}(C_{R,S})} \Big\}.
\end{align*}

\begin{defn}
A countable set $\Gamma= \{\gamma_{i,j} = (x_i,y_j): x_i\in\RR^n,\, y_j\in\RR, \hspace{1.5mm} i\in \ZZ^n,\, j\in\ZZ\}$ is said to be \textit{relatively separated set} with respect to both variables if 
\begin{align*}
A_{\Gamma}(\eta)=\sup_{x\in\RR^n} \suml_{j\in\ZZ^n} \chi_{B_{n}(x_i;\eta)}(x)<\infty \quad\text{and}\quad B_{\Gamma}(\eta)=\sup_{y\in\RR} \suml_{j\in\ZZ} \chi_{B_{1}(y_j;\eta)}(y)<\infty,
\end{align*}
for some $\eta>0.$ We say $\eta$ is the \textit{gap} for $\Gamma$ if
\begin{align*}
C_{\Gamma}(\eta)=\inf_{x\in\RR^n} \suml_{i\in\ZZ^n}\chi_{B_{n}(x_i;\eta)}(x)>1 \quad\text{and}\quad D_{\Gamma}(\eta)=\inf_{y\in\RR} \suml_{j\in\ZZ} \chi_{B_{1}(y_j;\eta)}(y)>1,
\end{align*}
where $B_{n}(x;\eta)$ denotes the open ball of radius $\eta$ center at $x$ in $\RR^n$ with respect to $\|\cdot\|_{\infty}$-norm.
\end{defn}
As the kernel $K$ satisfies the decay \eqref{kernel2} and the regularity condition \eqref{kernel0}, therefore, there exist a relatively separate set $\Gamma= \{\gamma_{i,j} = (x_i,y_j): x_i\in\RR^n,\, y_j\in\RR, \hspace{1.5mm} i\in\ZZ^n,\, j\in\ZZ\}$ with positive gap $\eta (<\frac{2}{n})$, and two families $\Phi = \{\phi_{\gamma}\}_{\gamma \in \Gamma} \subseteq L^{p,q}(\RR^{n+1})$ and $\tilde{\Phi}=\{\tilde{\phi}_{\gamma}\}_{\gamma \in \Gamma} \subseteq L^{p',q'}(\RR^{n+1})$ such that $f\in V$ can be reformulated by 
\begin{align} \label{frame1}
f = \sum_{\gamma \in \Gamma} \langle f,\tilde{\phi}_{\gamma} \rangle \phi_{\gamma},
\end{align}
see \cite{JS}. Moreover, the family $\tilde{\Phi}$ is a \textit{$(p,q)-$frame} for $V$, i.e., there exist positive constants $A$ and $B$ such that
\begin{align} \label{frame2}
A\|f\|_{L^{p,q}(\RR^{n+1})} \leq \big\|\{\langle f,\phi_{\gamma}\rangle\}_{\gamma \in \Gamma}\big\|_{\ell^{p,q}(\Gamma)} \leq B\|f\|_{L^{p,q}(\RR^{n+1})}, \hspace{3mm} \forall f \in V.
\end{align}
Let $N$ be a positive real number. We define the finite dimensional subspace $V_{N}$ of $V$ as 
\begin{align*}
V_N = \Big\{\sum_{\gamma \in \Gamma \cap [-\frac{N}{2}.\frac{N}{2}]^{n+1}} c_{\gamma}\phi_{\gamma}: c_{\gamma}\in\RR \Big\}.
\end{align*}
\section{Finite-Dimensional Estimate}
In this section, we prove that on the compact domain $C_{R,S}$, a function in $V$ is approximated by a function in $V_N$. We also show that the set of $\delta$-concentrated functions is totally bounded with respect to $\|\cdot\|_{L^{p,q}(C_{R,S})}$.
\begin{lemma}
For any given $\epsilon>0$ and $f \in V$ with $\|f\|_{L^{p,q}(\RR^{n+1})} = 1,$ we have 
\begin{align*}
\Big\|f - \sum_{\gamma \in \Gamma \cap [-\frac{N}{2},\frac{N}{2}]^{n+1}} \langle f,\tilde{\phi}_{\gamma} \rangle\phi_{\gamma} \Big\|_{L^{p,q}(C_{R,S})} < \epsilon
\end{align*}
where
\begin{align*}
N > R+S+\frac{2}{n}+ BC 4^{\frac{n}{p'}+\frac{1}{q'}+\frac{1}{2}}\Big( 2(4n^{-1}+S+1)^{nq}+(4n^{-1}+R+1)^q \Big)^{\frac{1}{q}} R^{\frac{n}{q}}S^{\frac{1}{q}} \epsilon^{-\frac{1}{n+2}}.
\end{align*}
\end{lemma}
\begin{proof}
The proof of this lemma is similar to that of Lemma 2.1 from \cite{PS}. Given $f \in V,$ we consider $f_{N} \in V_{N}$ by 
\begin{align} \label{lemma1.1}
f_{N} = \sum\limits_{\gamma \in \Gamma \cap [-\frac{N}{2},\frac{N}{2}]^{n+1}} \langle f,\tilde{\phi}_{\gamma} \rangle \phi_{\gamma}.
\end{align}
Let $(x,y)\in C_{R,S}$, then from equations \eqref{frame1} and \eqref{lemma1.1}, we have
\begin{align*}
&|f(x,y) - f_{N}(x,y)|\\
=& \Big|\sum_{\gamma \in \Gamma \setminus [-\frac{N}{2},\frac{N}{2}]^{n+1}} \langle f,\tilde{\phi}_{\gamma} \rangle \phi_{\gamma}(x,y)\Big| \\
\leq&  \sum_{|\lambda_2|>\frac{N}{2}} \sum_{\|\lambda_{1}\|_{\infty} >\frac{N}{2}} |c(\lambda_1,\lambda_2)\phi_{(\lambda_1,\lambda_2)}(x,y)| +  \sum_{|\lambda_2|>\frac{N}{2}} \sum_{\|\lambda_1\|_{\infty}\leq \frac{N}{2}} |c(\lambda_1,\lambda_2)\phi_{(\lambda_1,\lambda_2)}(x,y)|  \\
&\hspace{2.5cm}+  \sum_{|\lambda_2|\leq \frac{N}{2}} \sum_{\|\lambda_1\|_{\infty} >\frac{N}{2}} |c(\lambda_1,\lambda_2)\phi_{(\lambda_1,\lambda_2)}(x,y)| \\
:=& I_{1} + I_{2} + I_{3},
\end{align*}
where $(\lambda_{1},\lambda_{2}) = \gamma \in \Gamma$ and $c(\lambda_1,\lambda_2)=c(\gamma) = \langle f,\tilde{\phi}_{\gamma} \rangle$ are the frame constants. We estimate bound of each $I_i$, $i=1,2,3$ separately. Using H\"{o}lder inequality for mixed sequence space and equation \eqref{frame2} on $I_1$, we get
\begin{align*}
I_1&\leq  \bigg(\sum_{|\lambda_2|> \frac{N}{2}} \Big(\sum_{\|\lambda_1\|_{\infty} > \frac{N}{2}} |c(\lambda_1,\lambda_2)|^p \Big)^{\frac{q}{p}} \bigg)^{\frac{1}{q}} \bigg( \sum_{|\lambda_2| > \frac{N}{2}} \Big(\sum_{\|\lambda_1\|_{\infty} > \frac{N}{2}} |\phi_{(\lambda_1,\lambda_2)}(x,y)|^{p'} \Big)^{\frac{q'}{p'}} \bigg)^{\frac{1}{q'}} \\
&\leq B\bigg(\sum_{|\lambda_2|> \frac{N}{2}} \Big(\sum_{\|\lambda_1\|_{\infty} > \frac{N}{2}} |\phi_{(\lambda_1,\lambda_2)}(x,y)|^{p'} \Big)^{\frac{q'}{p'}} \bigg)^{\frac{1}{q'}} ,
\end{align*}
where explicit expression of $\phi_{(\lambda_{1},\lambda_{2})}$ is given bellow, see \cite{JS} for more details.
\begin{align*}
\phi_{(\lambda_{1},\lambda_{2})}(x,y) = \eta^{-\frac{1}{q}-\frac{n}{p}} \int_{-\frac{\eta}{2}}^{{\frac{\eta}{2}}} \int_{[-\frac{\eta}{2},\frac{\eta}{2}]^{n}} K(\lambda_1+z_1,\lambda_2+z_2,x,y)\,dz_1dz_2.
\end{align*}
Therefore, from \eqref{kernel2} we have
\begin{align*}
|\phi_{(\lambda_{1},\lambda_{2})}(x,y)|\leq& c\eta^{-\frac{1}{q}-\frac{n}{p}}\int_{[-\frac{\eta}{2},{\frac{\eta}{2}}]} \frac{1}{\big(1+|\lambda_2+z_2-y|\big)^{\beta}} \int_{[-\frac{\eta}{2},\frac{\eta}{2}]^{n}} \frac{1}{\big(1+\|\lambda_1+z_1-x\|_{1}\big)^{\alpha}}\,dz_{1}dz_{2} \\
\leq& c\eta^{-\frac{1}{q}-\frac{n}{p}}\int_{[-\frac{\eta}{2},{\frac{\eta}{2}}]} \frac{1}{\big(1+|\lambda_{2}-y|-\frac{\eta}{2}\big)^{\beta}} \int_{[-\frac{\eta}{2},\frac{\eta}{2}]^{n}} \frac{1}{\big(1+\|\lambda_{1}-x\|_{1}-\frac{n\eta}{2}\big)^{\alpha}}\,dz_{1}dz_{2} \\
\leq& c\eta^{\frac{1}{q'}+\frac{n}{p'}} \frac{1}{\big(1+|\lambda_{2}- y|-\frac{\delta}{2}\big)^{\beta}\big(1+\|\lambda_{1} - x\|_{1} - \frac{n\eta}{2}\big)^{\alpha}}.
\end{align*}
Hence,
\begin{align*}
&\sum_{\|\lambda_{1}\|_{\infty} > \frac{N}{2}}|\phi_{(\lambda_{1},\lambda_{2})}(x,y)|^{p'}\\
\leq& \Big(\frac{2}{\eta}\Big)^n c^{p'} \eta^{\frac{p'}{q'}+n} \sum_{\|\lambda_1\|_{\infty} > \frac{N}{2}} \frac{1}{\big(1+|\lambda_2-y|- \frac{\eta}{2}\big)^{\beta p'}\big(1+\|\lambda_1-x\|_1- \frac{n\eta}{2}\big)^{\alpha p'}}\Big(\frac{\eta}{2}\Big)^n \\
\leq& 2^{n}c^{p'}\eta^{\frac{p'}{q'}}\frac{1}{\big(1+|\lambda_2-y|- \frac{\eta}{2}\big)^{\beta p'}} \sum_{\|\lambda_{1}\|_{\infty} > \frac{N}{2}} \int_{B_{\frac{\eta}{2}}(\lambda_{1})}\frac{dz}{\big(1+\|z - x\|_{1}- \frac{n\eta}{2}\big)^{\alpha p'}},
\end{align*}
where $B_{\frac{\eta}{2}}(\lambda_{1}) = \{z = \big(z(1),\dots,z(n)\big)\in \RR^n: \lambda_1(i)-\frac{\eta}{2} \leq z(i) \leq \lambda_{1}(i),\, 1 \leq i \leq n \}$ and the collection $\{B_{\frac{\eta}{2}}(\lambda_{1})\}$ overlap at most $A_{\Gamma}(\eta)$. Therefore,
\begin{align*}
&\sum_{\|\lambda_{1}\|_{\infty} > \frac{N}{2}} |\phi_{(\lambda_{1},\lambda_{2})}(x,y)|^{p'}\\
\leq& 2^n c^{p'}\eta^{\frac{p'}{q'}}\frac{1}{\big(1+|\lambda_2-y|- \frac{\eta}{2}\big)^{\beta p'}} \int_{\RR^n \setminus [-\frac{N-\eta}{2},\frac{N-\eta}{2}]^n} \frac{A_{\Gamma}(\eta)\,dz}{\big(1+\|z-x\|_{1}- \frac{d\eta}{2}\big)^{\alpha p'}} \\
\leq& 2^n c^{p'}\eta^{\frac{p'}{q'}}\frac{1}{\big(1+|\lambda_2-y|- \frac{\eta}{2}\big)^{\beta p'}} \int_{\RR^n \setminus [-\frac{N-\eta-R}{2},\frac{N-\eta-R}{2}]^n} \frac{A_{\Gamma}(\eta)\,dz}{\big(1+\|z\|_{1}- \frac{n\eta}{2}\big)^{\alpha p'}} \\
\leq& 4^{n}c^{p'}\eta^{\frac{p'}{q'}}\frac{1}{\big(1+|\lambda_2-y|- \frac{\eta}{2}\big)^{\beta p'}} A_{\Gamma}(\eta)\int_{ [\frac{N-\eta-R}{2},\infty)^n}\frac{dz}{\|z\|_1^{\alpha p'}} \\
\leq& 4^{n}c^{p'}\eta^{\frac{p'}{q'}}\frac{1}{\big(1+|\lambda_2-y|- \frac{\eta}{2}\big)^{\beta p'}} \frac{A_{\Gamma}(\eta)}{\omega_{\alpha}\big(\frac{N-\eta- R}{2}n\big)^{\alpha p'-n}},
\end{align*}
where $\omega_{\alpha} = (\alpha p'-1) (\alpha p'-2)\cdots(\alpha p'-n).$
\begin{align*}
\bigg(\sum_{\|\lambda_{1}\|_{\infty} > \frac{N}{2}}|\phi_{(\lambda_{1},\lambda_{2})}(x,y)|^{p'}\bigg)^{\frac{q'}{p'}}\leq 4^{\frac{nq'}{p'}}c^{q'}\eta\frac{1}{\big(1+|\lambda_2-y|-\frac{\eta}{2}\big)^{\beta q'}} \frac{(A_{\Gamma}(\eta))^{\frac{q'}{p'}}}{\omega_{\alpha}^{\frac{q'}{p'}}\big(\frac{N-\eta- R}{2}\big)^{\alpha q'-\frac{nq'}{p'}}}.
\end{align*}
Similarly for summation over $\lambda_2$ we get
\begin{align*}
\sum_{|\lambda_2|> \frac{N}{2}}\bigg(\sum_{\|\lambda_{1}\|_{\infty}>\frac{N}{2}} |\phi_{(\lambda_{1},\lambda_{2})}(x,y)|^{p'}\bigg)^{\frac{q'}{p'}}\leq 4^{\frac{nq'}{p'}+1}c^{q'} \frac{(A_{\Gamma}(\eta))^{\frac{q'}{p'}}B_{\Gamma}(\eta)}{\omega_{\alpha}^{\frac{q'}{p'}}\omega_{\beta}\big(\frac{N-\eta-R}{2}\big)^{\alpha q'-\frac{nq'}{p'}}\big(\frac{N-\eta-S}{2}\big)^{\beta q'-1}},
\end{align*}
where $\omega_{\beta} = (\beta q' - 1).$\par

Let $D_1= cA_{\Gamma}(\eta)^{\frac{1}{p'}}B_{\Gamma}(\eta)^{\frac{1}{q'}}$ and using the fact that $\eta<\frac{2}{n}$, we have
\begin{align*}
I_1\leq BD_1 4^{\frac{n}{p'}+\frac{1}{q'}} \frac{1}{\omega_{\alpha}^{\frac{1}{p'}}\omega_{\beta}^{\frac{1}{q'}} \big(\frac{N-R}{2}-\frac{1}{n}\big)^{\alpha-\frac{n}{p'}} \big(\frac{N-S}{2}-\frac{1}{n}\big)^{\beta-\frac{1}{q'}}}.
\end{align*}

For $I_{2},$ similar lines of proof yield

\begin{align*}
\sum_{\|\lambda_{1}\|_{\infty} \leq \frac{N}{2}}|\phi_{(\lambda_{1},\lambda_{2})}(x,y)|^{p'}\leq& 2^{n}c^{p'}\eta^{\frac{p'}{q'}}\frac{1}{\big(1+|\lambda_2-y|- \frac{\eta}{2}\big)^{\beta p'}}\sum_{\|\lambda_{1}\|_{\infty} \leq \frac{N}{2}} \int\limits_{B_{\frac{\eta}{2}}(\lambda_1)}\frac{dz}{\big(1+\|z-x\|_{1}- \frac{n\eta}{2}\big)^{\alpha p'}}\\
\leq& 2^{n}c^{p'}\eta^{\frac{p'}{q'}}\frac{1}{\big(1+|\lambda_2-y|- \frac{\eta}{2}\big)^{\beta p'}} \int\limits_{[-\frac{N+\eta}{2},\frac{N+\eta}{2}]^n}\frac{A_{\Gamma}(\eta)\,dz}{\big(1+\|z-x\|_{1}- \frac{n\eta}{2}\big)^{\alpha p'}} \\
\leq& 4^{n}c^{p'}\eta^{\frac{p'}{q'}}\frac{1}{\big(1+|\lambda_2-y|-\frac{\eta}{2}\big)^{\beta p'}} \frac{A_{\Gamma}(\eta)\big(\frac{N+\eta}{2}\big)^n}{\big(1-\frac{n\eta}{2}\big)^{\alpha p'}}
\end{align*}
Hence
\begin{align*}
\bigg(\sum_{\|\lambda_1\|_{\infty} \leq \frac{N}{2}}|\phi_{(\lambda_{1},\lambda_{2})}(x,y)|^{p'}\bigg)^{\frac{q'}{p'}}\leq 4^{\frac{nq'}{p'}}c^{q'}\eta \frac{1}{\big(1+|\lambda_2-y|-\frac{\eta}{2}\big)^{\beta q'}} \frac{A_{\Gamma}(\eta)^{\frac{q'}{p'}}\big(\frac{N+\eta}{2}\big)^{\frac{nq'}{p'}}}{\big(1-\frac{n\eta}{2}\big)^{\alpha q'}}.
\end{align*}
Taking summation on $\lambda_{2}$ we get
\begin{align*}
\sum_{|\lambda_{2}|>\frac{N}{2}} \bigg(\sum_{\|\lambda_{1}\|_{\infty}\leq \frac{N}{2}} |\phi_{(\lambda_{1},\lambda_{2})}(x,y)|^{p'}\bigg)^{\frac{q'}{p'}}\leq& 4^{\frac{nq'}{p'}+1}c^{q'} \frac{A_{\Gamma}(\eta)^{\frac{q'}{p'}}\big(\frac{N+\eta}{2}\big)^{\frac{nq'}{p'}}}{\big(1-\frac{n\eta}{2}\big)^{\alpha q'}} \frac{B_{\Gamma}(\eta)}{\omega_{\beta}(\frac{N-\eta-S}{2})^{\beta q'-1}} \\
\leq& 4^{\frac{nq'}{p'}+1}c^{q'} \frac{A_{\Gamma}(\eta)^{\frac{q'}{p'}}(2\eta+S+1)^n}{\big(1-\frac{n\eta}{2}\big)^{\alpha q'}} \frac{B_{\Gamma}(\eta)}{\omega_{\beta}\big(\frac{N-\eta-S}{2}\big)^{\beta q'-n-1}},
\end{align*}
whenever $N>S+\eta+1.$\par

Let $D_2=\frac{cA_{\Gamma}(\eta)^{\frac{1}{p'}}B_{\Gamma}(\eta)^{\frac{1}{q'}}}{\big(1-\frac{n\eta}{2}\big)^{\alpha}}$, then $$I_2\leq BD_2 4^{\frac{n}{p'}+\frac{1}{q'}} \frac{(4n^{-1}+S+1)^n}{\omega_{\beta}^{\frac{1}{q'}}\big(\frac{N-S}{2}- \frac{1}{n}\big)^{\beta-\frac{n+1}{q'}}}.$$

In a similar way, we obtain $$I_3\leq BD_3 4^{\frac{n}{p'}+\frac{1}{q'}} \frac{(4n^{-1}+R+1)}{\omega_{\alpha}^{\frac{1}{p'}}\big(\frac{N-R}{2}-\frac{1}{n}\big)^{\alpha-\frac{n}{p'}-\frac{1}{q'}}},$$
where $D_3=\frac{cA_{\Gamma}(\eta)^{\frac{1}{p'}}B_{\Gamma}(\eta)^{\frac{1}{q'}}}{\big(1-\frac{\eta}{2}\big)^{\beta}}$ and $N> R+\eta + 1$\\

Let $C = \max\{D_{1},D_{2},D_{3}\}.$ Therefore,
\begin{align*}
\|f - f_{N}\|_{L^{p,q}(C_{R,S})}^{q}\leq& \big(BC 4^{\frac{n}{p'}+\frac{1}{q'}}\big)^q \Bigg( \frac{R^nS}{\omega_{\alpha}^{\frac{q}{p'}}\omega_{\beta}^{\frac{q}{q'}}\big(\frac{N - R}{2}-\frac{1}{n}\big)^{\alpha q-\frac{nq}{p'}}\big(\frac{N-S}{2}-\frac{1}{n}\big)^{\beta q-\frac{q}{q'}}} \\
&\hspace{1in}+\frac{R^nS(4n^{-1}+S+1)^{nq}}{\omega_{\beta}^{\frac{q}{q'}}\big(\frac{N-S}{2}-\frac{1}{n}\big)^{\beta q-\frac{(n+1)q}{q'}}}+ \frac{R^nS(4n^{-1}+R+1)^q}{\omega_{\alpha}^{\frac{q}{p'}}\big(\frac{N-R}{2}-\frac{1}{n} \big)^{\alpha q-\frac{nq}{p'}-\frac{q}{q'}}} \Bigg) \\
\leq& \big(BC 4^{\frac{n}{p'}+\frac{1}{q'}}\big)^q \Bigg( \frac{2R^nS(4n^{-1}+S+1)^{nq}}{\omega_{\beta}^{\frac{q}{q'}}\big(\frac{N-S}{2}-\frac{1}{n}\big)^{\beta q-\frac{(n+1)q}{q'}}}+ \frac{R^nS(4n^{-1}+R+1)^q}{\omega_{\alpha}^{\frac{q}{p'}}\big(\frac{N-R}{2}-\frac{1}{n} \big)^{\alpha q-\frac{nq}{p'}-\frac{q}{q'}}}\Bigg)\\
\leq& \big(BC 4^{\frac{n}{p'}+\frac{1}{q'}}\big)^q R^nS \Big( 2(4n^{-1}+S+1)^{nq}+(4n^{-1}+R+1)^q \Big)\\
&\hspace{2.5in}\times \Big( \frac{N-R-S}{2}-\frac{1}{n} \Big)^{-(n+2)q} \\
<& \epsilon^{q}
\end{align*}
whenever
\begin{align*}
N > R+S+\frac{2}{n}+ BC 4^{\frac{n}{p'}+\frac{1}{q'}+\frac{1}{2}}\Big( 2(4n^{-1}+S+1)^{nq}+(4n^{-1}+R+1)^q \Big)^{\frac{1}{q}} R^{\frac{n}{q}}S^{\frac{1}{q}} \epsilon^{-\frac{1}{n+2}}.
\end{align*}
This completes the proof.
\end{proof}
We recall the following result on estimation of covering number of closed and bounded ball in a finite dimensional normed space.
\begin{prop}[\cite{CZ}] \label{lemmanumber}
Let $Y$ be a Banach space of dimension $s$ and $\overline{B(0;r)}$ denotes the closed ball of radius $r$ center at origin in $Y$. The minimum number of open balls of radius $r_1$ to cover $\overline{B(0;r)}$ is bounded by $\big(\frac{2r}{r_1} + 1\big)^s.$
\end{prop}
The following results for mixed Lebesgue space $L^{p,q}(\RR^{n+1})$ is similar to the classical Lebesgue space $L^p(\RR^n)$, see \cite{PS}.
\begin{lemma}
If $f\in V(R,S,\delta):=\{ f\in V^{\star}(R,S,\delta):\|f\|_{L^{p,q}(\RR^{n+1})}=1 \}$ then
\begin{align*}
\|f\|_{L^{\infty,\infty}(C_{R,S})}\leq D\|f\|_{L^{p,q}(C_{R,S})},
\end{align*} 
where $D=\frac{\sup\limits_{(x,y) \in C_{R,S}} \|K(x,y,\cdot,\cdot)\|_{L^{p',q'}(\RR^{n+1})}}{(1-\delta)^{\frac{1}{q}}}.$
\end{lemma}
\begin{lemma} \label{lemmabounded}
The set $V(R,S,\delta)$ is totally bounded with respect to $\|\cdot\|_{L^{\infty,\infty}(C_{R,S})}.$ 
\end{lemma}
\begin{rem} \label{remarknumber} 
\-
\begin{enumerate}[label=(\roman*)]
\item In the above lemma, we choose $f_N\in V_N$ such that $\|f-f_N\|_{L^{\infty,\infty}(C_{R,S})}<\frac{\epsilon}{2}$ and  
\begin{align*}
N=R+S+2+BC 4^{\frac{n}{p'}+\frac{1}{q'}+1}\Big( 2(4n^{-1}+S+1)^{nq}+(4n^{-1}+R+1)^q \Big)^{\frac{1}{q}} \epsilon^{-\frac{1}{n+2}}.
\end{align*}
Then dimension of $V_{N}$ is bounded by 
\begin{align*}
N^{n+1} N_{0}(\Gamma) \leq 2^{n+1}N_{0}(\Gamma)\Big[(R+S+2)^{n+1}+C_{1}\epsilon^{-\frac{n+1}{n+2}} \Big]:= d_{\epsilon}
\end{align*} 
where $N_{0}(\Gamma) = \sup\limits^{}_{k\in \ZZ^{n+1}} \Big(k+[-\frac{1}{2},\frac{1}{2}]^{n+1}\Big) \cap \Gamma$ and 
\begin{align*}
  C_1= \bigg( BC 4^{\frac{n}{p'}+\frac{1}{q'}+1}\Big( 2(4n^{-1}+S+1)^{nq}+(4n^{-1}+R+1)^q \Big)^{\frac{1}{q}} \bigg)^{n+1}.
\end{align*}
Note that, the choice of $N$ is not unique. One can also consider a different value of $N$ according to the bound of $N$.

\item Let $\aa(\epsilon)$ be a $\epsilon$-net of the totally bounded set $V(R,S,\delta)$ and $N(\epsilon)$ be the number of element in $\aa(\epsilon)$. The set $\aa(\epsilon)$ is same as $\frac{\epsilon}{2}$-net of $\overline{B(0,D+\frac{\epsilon}{2})}$ in $V_{N}$ with respect to $\|\cdot\|_{L^{\infty,\infty}(C_{R,S})}$. Hence, Proposition \ref{lemmanumber} implies $$N(\epsilon) \leq \exp\Big( d_{\epsilon} \log\Big(\frac{8D}{\epsilon}\Big) \Big).$$

\end{enumerate}
\end{rem}


\section{Random Sampling}

In this section, we define independent random variable from uniformly distributed random sample set on $C_{R,S}.$ Using the Bernstein inequality on the independent random variables, we prove the random sampling result.

Let $\{(x_{i},y_{j}):i,j \in \NN \}$ be a sequence of independent and identically distributed random variables uniformly distributed over the cube $C_{R,S}.$ For every $f \in V,$ we define the random variable 
\begin{align*}
Z_{i,j}(f) = |f(x_i,y_j)|- \frac{1}{R^nS}\iint_{C_{R,S}}|f(x,y)|\,dxdy
\end{align*}
with expectation $E\big( Z_{i,j}(f) \big)=0$.\par

To derive properties of $Z_{i,j}(f)$, first we recall the following results.
\begin{lemma}[\cite{JL}] \label{lemma 4.1}
	For $1\leq p,q\leq \infty$ and $X=\{ (x_i,y_j): 1\leq i\leq l,\, 1\leq j\leq m \}$, we have
	\begin{gather*}
	\|\{f(x_{i},y_{j})\}\|_{\ell^{p,q}(X)} \leq \sum\limits_{j=1}^{m}\sum\limits_{i=1}^{l}|f(x_{i},y_{j})| 
	\leq l^{1-\frac{1}{p}} m^{1-\frac{1}{q}}\|\{f(x_{i},y_{j})\}\|_{\ell^{p,q}(X)},\\
	\|f\|_{L^{1,1}(C_{R,S})} \leq R^{n-\frac{n}{p}}S^{1-\frac{1}{q}}\|f\|_{L^{p,q}(C_{R,S})},\\
	\|f\|_{L^{p,q}(C_{R,S})}^{pq} \leq R^{(p-1)n}S^{q-1}D^{pq-1}\|f\|_{L^{1,1}(C_{R,S})}.
	\end{gather*}
\end{lemma}

\begin{lemma} \label{lemmarandom}
For $f,g \in V$ with $\|f\|_{L^{p,q}(\RR^{n+1})} = \|g\|_{L^{p,q}(\RR^{n+1})}=1,$ then the following inequalities holds:
\begin{enumerate}[label=(\alph*)]
\item $Var\big( Z_{i,j}(f) \big) \leq \frac{1}{R^{\frac{n}{p}}S^{\frac{1}{q}}}\sup\limits_{(x,y)\in \RR^{n+1}} \|K(x,y, \cdot,\cdot)\|_{L^{p'q'}(\RR^{n+1})}.$ 

\item $\big\| Z_{i,j}(f) \big\|_{\ell^{\infty,\infty}} \leq \sup\limits_{(x,y)\in \RR^{n+1}} \|K(x,y,\cdot,\cdot)\|_{L^{p'q'}(\RR^{n+1})}.$

\item $Var\big( Z_{i,j}(f) - Z_{i,j}(g) \big) \leq \frac{2}{R^{\frac{n}{p}}S^{\frac{1}{q}}}\|f-g\|_{L^{\infty,\infty}(C_{R,S})}.$

\item $\big\| Z_{i,j}(f)- Z_{i,j}(g) \big\|_{\ell^{\infty,\infty}} \leq  \|f-g\|_{L^{\infty,\infty}(C_{R,S})}.$
\end{enumerate}
\end{lemma}

\begin{proof}
\begin{enumerate}[label=(\alph*)]
	\item Lemma \ref{lemma 4.1} and reproducing properties of function in $V$ implies
	\begin{align*}
	Var\big( Z_{i,j}(f) \big) &= E\Big( \big( |f(x_i,y_j)| - E\big(|f(x_i,y_j)|\big) \big)^{2} \Big) \\
	&=E\big( |f(x_i,y_j)|^{2} \big) - \big( E\big(|f(x_i,y_j)|\big) \big)^{2} \\
	&\leq E\big(|f(x_i,y_j)|^{2}\big)  \\
	&= \frac{1}{R^nS} \iint_{C_{R,S}}|f(x,y)|^{2}\,dxdy \\
	&\leq \frac{1}{R^nS}\|f\|_{L^{\infty,\infty}(C_{R,S})} \|f\|_{L^{1,1}(C_{R,S})} \\
	&\leq \frac{1}{R^{\frac{n}{p}}S^{\frac{1}{q}}}\sup\limits_{(x,y)\in \RR^{n+1}}\|K(x, y,\cdot,\cdot)\|_{L^{p'q'}(\RR^{n+1})}.
	\end{align*}
	
	\item Reproducing properties of function in $V$ implies
	\begin{align*}
	\big\| Z_{i,j}(f) \big\|_{\ell^{\infty,\infty}} &=\sup\limits_{i,j\in \NN} \Big\{\big| |f(x_i,y_j)|-\frac{1}{R^nS} \iint_{C_{R,S}}|f(x,y)|\,dxdy \big|\Big\} \\
	&\leq \max\Big\{ \|f\|_{L^{\infty,\infty}(C_{R,S})},\frac{1}{R^nS}\|f\|_{L^{1,1}(C_{R,S})} \Big\} \\
	&\leq \sup\limits_{(x,y)\in \RR^{n+1}}\|K(x,y,\cdot,\cdot)\|_{L^{p'q'}(\RR^{n+1})}.
	\end{align*}
	
	\item Similar argument as in part (a) implies
	\begin{align*}
	Var\big( Z_{i,j}(f)-Z_{i,j}(g) \big)=& E\Big( \big(|f(x_i,y_j)| - |g(x_i,y_j)| \big)^{2} \Big) - \Big( E\big(|f(x_i,y_j)| - |g(x_i,y_j)|\big) \Big)^2 \\
	=& \frac{1}{R^nS} \iint_{C_{R,S}} \big||f(x,y)|-|g(x,y)|\big| \big(|f(x,y)| + |g(x,y)|\big)\,dxdy \\
	\leq& \frac{2}{R^{\frac{n}{p}}S^{\frac{1}{q}}}\|f-g\|_{L^{\infty,\infty}(C_{R,S})}.
	\end{align*}
	
	\item Similar argument as in part (b) implies
	\begin{align*}
	&\|Z_{i,j}(f)-Z_{i,j}(g)\|_{\ell^{\infty,\infty}} \\
	=& \sup\limits_{i,j\in \NN}\Big\{ \big||f(x_i,y_j)|-|g(x_i,y_j)|-\frac{1}{R^nS} \iint_{C_{R,S}} \big( |f(x,y)|-|g(x,y)| \big)\,dxdy\big| \Big\} \\
	\leq& \max\Big\{ \|f-g\|_{L^{\infty,\infty}(C_{R,S})}, \frac{1}{R^nS}\|f-g\|_{L^{1,1}(C_{R,S})} \Big\} \\
	\leq& \|f-g\|_{L^{\infty,\infty}(C_{R,S})}.
	\end{align*}
\end{enumerate}
This completes the proof.
\end{proof}

In the rest of the article, we denote $k =  \sup\limits_{(x,y)\in \RR^{n+1}} \|K(x,y,\cdot,\cdot)\|_{L^{p'q'}(\RR^{n+1})}.$ We recall the Bernstein's inequality for independent random variable with zero mean.

\begin{theorem}[\cite{B}]\label{BI}
Let $\big\{ Y_{i,j}: 1\leq i\leq l,1\leq j\leq m \big\}$ be a sequence of independent random variable with expectation $E(Y_{i,j})=0,\,\forall\, i,j.$ Assume that $Var\big(Y_{i,j}\big)\leq \sigma^2$ and $|Y_{i,j}|\leq M$ for all $i,j.$ Then for $\lambda \geq 0$ we have
$$P\bigg(\Big|\sum_{j=1}^m\sum_{i=1}^l Y_{i,j} \Big|\geq \lambda \bigg) \leq 2\exp\Big( -\frac{\lambda^2}{2lm\sigma^2+\frac{2}{3}M\lambda} \Big).$$
\end{theorem}

\begin{lemma} \label{lemma3}
Let $\{(x_i,y_j):i,j\in \NN \}$ be a sequence of i.i.d random variables that are drawn uniformly from $C_{R,S}.$ Then
$$P\bigg(\sup_{f \in V(R,S,\delta)} \Big|\sum_{j=1}^m\sum_{i=1}^l Z_{i,j}(f)\Big| \geq \lambda \bigg) \leq 3a\exp\bigg(-b\frac{\lambda^2}{12lmR^{-\frac{n}{p}}S^{-\frac{1}{q}} + \lambda}\bigg),$$ where $a=\exp\big(G(R+S)^{n^2+n}\big)$ and $b = \min\Big\{\sqrt{2}C_{2},\frac{3}{4k}\Big\}$, for some positive constant $G, C_2$.
\end{lemma}

\begin{proof}
We divide the proof in four parts. \par
Step-$1.$ Let $f \in V(R,S,\delta)$. From Lemma \ref{lemmabounded}, we can construct a sequence $\{f_r\}_{r\in \NN}$ such that $f_r \in \aa(2^{-r})$ and $\|f-f_r\|_{L^{\infty,\infty}(C_{R,S})}<2^{-r}.$ Then, random variable $Z_{i,j}(f)$ can be written as
$$Z_{i,j}(f) = Z_{i,j}(f_{1})+\sum_{r=2}^{\infty} \left(Z_{i,j}(f_{r}) - Z_{i,j}(f_{l-1})\right),$$
due to the fact that the partial sum $s_{r'}(f) = Z_{i,j}(f_{1}) + \sum\limits_{r=2}^{r'}\left(Z_{i,j}(f_{r}) - Z_{i,j}(f_{r-1})) = Z_{i,j}(f_{r'}\right)$ and 
\begin{align*}
\|Z_{i,j}(f) - Z_{i,j}(f_{r'})\|_{\ell^{\infty,\infty}} \leq \|f-f_{r'}\|_{L^{\infty,\infty}(C_{R,S})}\to 0 \text{ when } r' \to \infty.
\end{align*}
We consider the event 
$$\ee= \bigg\{ \sup_{f\in V(R,S,\delta)} \Big|\sum_{j=1}^m\sum_{i=1}^l Z_{i,j}(f)\Big| \geq \lambda \Big\}, ~~~\ee_1=\Big\{ \exists \, f_1\in \aa\Big(\frac{1}{2}\Big): \Big|\sum_{j=1}^m\sum_{i=1}^l Z_{i,j}(f_1)\Big| \geq \frac{\lambda}{2}\Big\}, $$
and for $r \geq 2,$ we define
\begin{multline*}
\ee_r= \Big\{ \exists\, f_{r} \in \aa(2^{-r}) \text{ and } f_{r-1} \in \aa(2^{-r+1})\text{ with }\\ \|f_{r}-f_{r-1}\|_{L^{\infty,\infty}(C_{R,S})} 
\leq 3\cdot2^{-r}:\Big|\sum_{j=1}^m\sum_{i=1}^l Z_{i,j}(f_{r})-Z_{i,j}(f_{r-1})\Big| \geq \frac{\lambda}{2r^{2}}\Big\}. 
\end{multline*}
 
If the independent random sample set $\{(x_i,y_j):i,j\in \NN\}$ satisfy $$\sup_{f\in V(R,S,\delta)} \Big|\sum_{j=1}^m\sum_{i=1}^l Z_{i,j}(f)\Big| \geq \lambda,$$ then at least one of the event $\ee_r$, $r\geq 1$ holds. Therefore, $\ee \subseteq \cup_{r=1}^{\infty} \ee_r$.
\par
Step-$2.$ Theorem \ref{BI} for the random variable $\{ Z_{i,j}(f_1):1\leq i\leq l,1\leq j\leq m \}$ and Lemma \ref{lemmarandom} implies,
\begin{align*}
P\bigg( \Big|\sum_{j=1}^m\sum_{i=1}^l Z_{i,j}(f_1)\Big|\geq \frac{\lambda}{2} \Big) &\leq 2\exp\Bigg(-\frac{\frac{\lambda^{2}}{4}}{2lmR^{-\frac{n}{p}}S^{-\frac{1}{q}}k+ \frac{1}{3}k\lambda}\Bigg)\\
&\leq 2\exp\Bigg(-\frac{3}{4k}\frac{\lambda^{2}}{6lmR^{-\frac{n}{p}}S^{-\frac{1}{q}} + \lambda}\Bigg). 
\end{align*}
Hence
\begin{align*}
P(\ee_1)\leq 2N\Big(\frac{1}{2}\Big) \exp\Bigg( -\frac{3}{4k}\frac{\lambda^{2}}{6lmR^{-\frac{n}{p}}S^{-\frac{1}{q}}+\lambda} \Bigg).
\end{align*} \par

Step-$3.$ We can estimate the probability of the event $\ee_r$ in the similar way as in previous step. Theorem \ref{BI} and Lemma \ref{lemmarandom} implies,
\begin{align*}
&P\Big(\Big|\sum_{j=1}^m\sum_{i=1}^l \big(Z_{i,j}(f_{r})-Z_{i,j}(f_{r-1})\big)\Big|\geq \frac{\lambda}{2r^2}\Big) \\
\leq&  2\exp\Bigg( -\frac{\frac{\lambda^{2}}{4r^4}}{4lmR^{-\frac{n}{p}}S^{-\frac{1}{q}}\|f_r-f_{r-1}\|_{L^{\infty,\infty}(C_{R,S})} + \frac{2}{3}\frac{\lambda}{2r^2}\|f_r-f_{r-1}\|_{L^{\infty,\infty}(C_{R,S})}}\Bigg) \\
\leq&  2\exp\Bigg(-\frac{2^r}{4r^4}\frac{\lambda^{2}}{12lmR^{-\frac{n}{p}}S^{-\frac{1}{q}} + \lambda}\Bigg). 
\end{align*}
From Remark \ref{remarknumber}, we have
\begin{align*}
N(\epsilon) \leq \exp\Big(2^{n+1}N_{0}(\Gamma)\Big[(R+S+2)^{n+1}+C_{1}\epsilon^{-\frac{n+1}{n+2}} \Big] \log\Big(\frac{8D}{\epsilon}\Big) \Big).
\end{align*}
Therefore,
\begin{align*}
N(2^{-r})\leq& \exp\left(2^{n+1}N_0(\Gamma)\left[(R+S+2)^{n+1}+C_1 2^{r\frac{n+1}{n+2}}\right]\big[(r+3)\log 2+\log D \big] \right) \\
N(2^{-r+1}) \leq& \exp\left(2^{r+1}N_0(\Gamma)\left[(R+S+2)^{n+1}+C_1 2^{(r-1)\frac{n+1}{n+2}}\right]\big[(r+2)\log 2+\log D\big] \right).
\end{align*}
Hence,
\begin{align*}
N(2^{-r})N(2^{-r+1})\leq \exp\Big(2^{n+1}N_0(\Gamma)\left[(R+S+2)^{n+1}+C_1 2^{r\frac{n+1}{n+2}}\right]\big[ (2r+5)\log 2 + 2\log D \big] \Big).
\end{align*}
Therefore,
\begin{align*}
P(\ee_{r})\leq& 2\exp\Bigg(2^{n+1}N_0(\Gamma)\Big[(R+S+2)^{n+1}+C_1 2^{r\frac{n+1}{n+2}}\Big]\big[(2r+5)\log 2 + 2\log D\big] \\
&\hspace{3in}-\frac{2^{r}}{4r^4}\frac{\lambda^{2}}{12lmR^{-\frac{n}{p}}S^{-\frac{1}{q}} + \lambda} \Bigg)\\
\leq&  2\exp\Bigg[2^{r\frac{n+2}{n+3}}\Bigg(2^{n+1}N_0(\Gamma)\Big[(R+S+2)^{n+1}2^{-r\frac{n+2}{n+3}}+ C_1 2^{-\frac{r}{(n+2)(n+3)}}\Big]\\
&\hspace{1.25in}\times\Big[(2r+5)\log 2 + \log D\Big]-\frac{2^{\frac{r}{n+3}}}{4r^4} \frac{\lambda^{2}}{12lmR^{-\frac{n}{p}}S^{-\frac{1}{q}} + \lambda} \Bigg)\Bigg] \\
\leq& 2\exp\Bigg[2^{r\frac{n+2}{n+3}}\Bigg(2^{n+1}N_0(\Gamma)\Big[(R+S+2)^{n+1}(4+\log D)+C_1\big(4(n+2)(n+3)+\log D\big)\Big]\\ 
&\hspace{3in}-\frac{2^{\frac{1}{4\log 2}}}{4\left(\frac{n+3}{4\log 2}\right)^{4}} \frac{\lambda^{2}}{12lmR^{-\frac{n}{p}}S^{-\frac{1}{q}}+\lambda} \Bigg)\Bigg].
\end{align*}
Let 
\begin{align*}
C_2=& \frac{2^{\frac{1}{4\log 2}-6}(\log 2)^4}{(n+3)^4} \\
C_3=& 2^{n+1}N_0(\Gamma)\Big[(R+S+2)^{n+1}(4+\log D)+C_1\big(4(n+2)(n+3)+\log D\big)\Big] \\
\phi=& \frac{\lambda^{2}}{12lmR^{-\frac{n}{p}}S^{-\frac{1}{q}}+\lambda}.
\end{align*}
Then $P(\ee_r)\leq 2\exp\big(-2^{r\frac{n+2}{n+3}}(C_2\phi-C_3)\big)$, for large $\lambda$, $C_2\phi-C_3> 0.$ \par

Step-$4.$ Since $\ee\subseteq \cup_{r=1}^{\infty} \ee_r,$ we get $\displaystyle P(\ee)\leq \sum_{r=1}^{\infty} P(\ee_r)$. For $u,v>0$, the series $\sum_{r=2}^{\infty} e^{-u^r v}\leq \frac{1}{uv\log u}e^{-uv}$, therefore,
\begin{align*}
\sum_{r=2}^{\infty} P(\ee_r)\leq& \frac{2^{\frac{1}{n+3}}(n+3)}{(n+2)(C_2\phi-C_3)\log 2} \exp(-2^{\frac{n+2}{n+3}}(C_2\phi-C_3))\\
\leq& \frac{6}{C_2\phi-C_3}\exp\big( -\sqrt{2}(C_2\phi-C_3) \big)
\end{align*}
We choose $\lambda$ large enough such that $(C_2\phi-C_3)\geq 6.$ Hence
\begin{align*}
\sum_{r=2}^{\infty} P(\ee_r) \leq \exp(\sqrt{2}C_3)\exp\Bigg( -\sqrt{2}C_2 \frac{\lambda^2}{12lmR^{-\frac{n}{p}}S^{-\frac{1}{q}}+\lambda} \Bigg).
\end{align*}
Let $a_1= \max\big\{\exp(\sqrt{2}C_3),N\left(\frac{1}{2}\right)\big\}$ and $b = \min\big\{\sqrt{2}C_2,\frac{3}{4k}\big\}.$ Then
\begin{align*}
P(\ee)\leq 3a_1\exp\Bigg(-b\frac{\lambda^2}{12lmR^{-\frac{n}{p}}S^{-\frac{1}{q}}+ \lambda}\Bigg).
\end{align*}
To compute the bound of the constant $a_1$, consider
\begin{align*}
\exp(\sqrt{2}C_3)=& \exp\Big(2^{n+1}N_0(\Gamma)\Big[(R+S+2)^{n+1}(4+\log D)+C_1\big(4(n+2)(n+3)+\log D\big)\Big]\Big) \\
\leq& \exp\Big(2^{n+1}N_0(\Gamma)[(R+S+2)^{n+1}+C_1]\big(4(n+2)(n+3)+\log D\big)\Big),\\
N\Big(\frac{1}{2}\Big)\leq& \exp\Big( 2^{n+1}N_0(\Gamma)\left[(R+S+2)^{n+1}+2C_1\right]\log 16D\Big)\\
\leq & \exp\Big(2^{n+1}N_0(\Gamma)[(R+S+2)^{n+1}+2C_1 ]\big(4(n+2)(n+3)+\log D\big)\Big),
\end{align*}
and $\displaystyle C_1= \bigg( BC 4^{\frac{n}{p'}+\frac{1}{q'}+1}\Big( 2(4n^{-1}+S+1)^{nq}+(4n^{-1}+R+1)^q \Big)^{\frac{1}{q}} \bigg)^{n+1}.$\par 

Hence, we have $a_1\leq \exp\big(G(R+S)^{n^2+n}\big):= a$, where $R,S\geq 1$ and
\begin{align*}
 G=2^{n+1}N_0(\Gamma)\Big(4(n+2)(n+3)+\log D\Big) \Big( 1+2\big( 3BC 4^{\frac{n}{p'}+\frac{1}{q'}+1} \big)^{n+1} \Big).
\end{align*}
This complete the proof.
\end{proof}

\begin{theorem}  \label{Mainresult}
Let $1 \leq p,q<\infty.$ Suppose $\{(x_{i},y_{j}) : i,j \in \mathbb{N}\}$ is a sequence of independent random variables that are uniformly distributed over the cube $C_{R,S}.$ Then for any $0 <\mu <1$, the sampling inequality
\begin{multline} \label{samplinginequality}
\frac{l^{\frac{1}{p}}m^{\frac{1}{q}}(1-\delta)^{pq}D^{1-pq}}{R^{np}S^q} (1-\mu)\|f\|_{L^{p,q}(\RR^{n+1})}\leq\Big( \sum_{j=1}^m \Big( \sum_{i=1}^l |f(x_i,y_j)|^p \Big)^{\frac{q}{p}} \Big)^{\frac{1}{q}}\\
\leq lm\big(1 + \mu D^{1-pq} \big)\|f\|_{L^{p,q}(\RR^{n+1})}
\end{multline}
holds uniformly for all $f \in V^{\star}(R,S,\delta)$ with probability at least 
\begin{align*}
1-3a\exp\Bigg(-b\frac{\mu^2 lm(1-\delta)^{2pq}D^{2(1-pq)}}{12R^{2np}S^{2q}+ D^{1-pq}R^{np}S^q}\Bigg).
\end{align*}
Here, the positive constant $a$ and $b$ are same as in Lemma \ref{lemma3}.
\end{theorem}

\begin{proof}
Let $\{(x_{i},y_{j}) : i,j \in \mathbb{N}\}$ be an i.i.d. random samples uniformly drawn from the set $C_{R,S}.$ Then the event $\ee$ define in Lemma \ref{lemma3} is the complement of the event
\begin{align*}
\ee^c= \Big\{\sup_{f\in V(R,S,\delta)}\Big|\sum_{j=1}^m\sum_{i=1}^l Z_{i,j}(f)\Big|\leq \lambda \Big\}.
\end{align*}
The event $\ee^c$ is equivalent to
\begin{gather*}
\Big|\sum_{j=1}^m\sum_{i=1}^l |f(x_i,y_j)|-\frac{lm}{R^n S} \iint_{C_{R,S}} |f(x,y)|\,dxdy\Big| \leq \lambda\\
\frac{lm}{R^n S} \iint_{C_{R,S}} |f(x,y)|\,dxdy-\lambda \leq \sum_{j=1}^m\sum_{i=1}^l |f(x_i,y_j)|\leq \frac{lm}{R^n S} \iint_{C_{R,S}} |f(x,y)|\,dxdy+\lambda.
\end{gather*}
Lemma \ref{lemma 4.1} and $f\in V(R,S,\delta)$ yields
\begin{align*}
l^{\frac{1}{p}-1}m^{\frac{1}{q}-1} \bigg( \frac{lm(1-\delta)^{pq}D^{1-pq}}{R^{np}S^q}- \lambda \bigg)\leq \Big( \sum_{j=1}^m \Big( \sum_{i=1}^l |f(x_i,y_j)|^p \Big)^{\frac{q}{p}} \Big)^{\frac{1}{q}} \leq \frac{lm}{R^{\frac{n}{p}}S^{\frac{1}{q}}}+ \lambda.
\end{align*}
Consider $\displaystyle \lambda= \frac{\mu lm(1-\delta)^{pq}D^{1-pq}}{R^{np}S^q},$ then for every $f\in V(R,S,\delta)$ we have
\begin{align*}
\frac{l^{\frac{1}{p}}m^{\frac{1}{q}}(1-\delta)^{pq}D^{1-pq}}{R^{np}S^q}(1-\mu)\leq \Big( \sum_{j=1}^m \Big( \sum_{i=1}^l |f(x_i,y_j)|^p \Big)^{\frac{q}{p}} \Big)^{\frac{1}{q}} \leq& \frac{lm}{R^{\frac{n}{p}}S^{\frac{1}{q}}}+\frac{\mu lm(1-\delta)^{pq}D^{1-pq}}{R^{np}S^q}\\
\leq& lm\big(1 + \mu D^{1-pq} \big).
\end{align*}

Hence the sampling inequality \eqref{samplinginequality} holds for all $f\in V^{\star}(R,S,\delta)$ with probability at least 
\begin{align*}
P(\ee^c)= 1-P(\ee)\geq& 1-3a \exp\Bigg(-b\frac{\mu^2 lm(1-\delta)^{2pq}D^{2(1-pq)}}{12R^{2np-\frac{n}{p}}S^{2q-\frac{1}{q}}+\mu D^{1-pq}(1-\delta)^{pq} R^{np}S^q}\Bigg)\\
\geq& 1-3a\exp\Bigg(-b\frac{\mu^2 lm(1-\delta)^{2pq}D^{2(1-pq)}}{12R^{2np}S^{2q}+ D^{1-pq}R^{np}S^q}\Bigg).
\end{align*}
This complete the proof.
\end{proof}

\begin{eg}
For $(x,y) \in \RR^n\times \RR,$ we define
$$\varphi(x,y) = \varphi(x(1),\dots,x(n),y)= \Big( \frac{3}{2} \Big)^{\frac{n}{2}} \frac{3}{\sqrt{n^2-3n-2}} \max\Big\{1-3\sum_{i=1}^n |x(i)|-3|y|,0\Big\}$$
with $supp(\varphi)\subseteq [-\frac{1}{3},\frac{1}{3}]^{n+1},$ and consider $\Lambda\subset \RR^{n+1}$ is a separated set with gap greater than or equal to $\frac{2}{3}$, that is, for any $v,v' \in \Lambda$ with $v \neq v', \inf_{v,v'\in \Lambda} \|v-v'\| \geq \frac{2}{3}.$ Let the kernel $K$ be defined by 
\begin{align*}
K(x,y,s,t) =\sum_{(\alpha,\beta) \in \Lambda} \varphi(x-\alpha,y-\beta)\varphi(s-\alpha,t-\beta).
\end{align*}
Then the integral operator associated to the kernel $K$ is idempotent.
Therefore, the space $\displaystyle V_{\varphi}:= \Big\{\sum_{(\alpha,\beta)\in \Lambda} c_{\alpha,\beta}\varphi(\cdot-\alpha,\cdot-\beta):c=(c_{\alpha,\beta})\in \ell^{p,q}(\Lambda) \Big\} \subseteq L^{p,q}(\RR^{n+1})$ is an image of an integral idempotent operator and the kernel $K$ satisfy the regularity condition \eqref{kernel0}. Further, there exists $C_n>0$ such that
$$|\varphi(x,y)| \leq C_n e^{-(\|x\|_{2}^{2}+|y|^{2})},$$
and
\begin{align*}
|K(x,y,s,t)|\leq& \Big( \frac{3}{2} \Big)^{\frac{n}{2}} \frac{3C_n}{\sqrt{n^2-3n-2}} \sum_{(\alpha,\beta)\in \Lambda} e^{-(\|x-\alpha\|_{2}^2 + |y-\beta|^2)} e^{-(\|s-\alpha\|_{2}^2+|t-\beta|^2)} \\
\leq& \Big( \frac{3}{2} \Big)^{\frac{n}{2}} \frac{3C_n}{\sqrt{n^2-3n-2}} \sum_{(\alpha,\beta)\in \Lambda} e^{-\frac{1}{4}\big( \big(3\|x\|_{2}^2-2\|x\|_{2}\|s\|_{2}+3\|s\|_{2}^2\big)+\big(3|y|^2-2|y||t|+ 3|t|^2\big) \big)}  \\
& \hspace{6cm} \times e^{-2\big(\|\alpha\|_{2}-\frac{\|x\|_{2}+\|s\|_{2}}{2}\big)^{2}} e^{-2\big(|\beta|-\frac{|y|+|t|}{2}\big)^{2}} \\
\leq&  \Big(\frac{1}{2}\Big)^{\frac{n}{2}}\frac{C_n S_n\sqrt{\pi}}{2} \Gamma\Big(\frac{n}{2}\Big)e^{-\frac{1}{4}\big(\|x-s\|_{2}+|y-t|_{2}\big)^2}.
\end{align*}
where $S_n$ represents the surface area of the unit sphere in $\RR^n$. This implies the kernel $K$ satisfy the decay condition \eqref{kernel2}. For given $\epsilon>0,$ the sampling inequality \eqref{samplinginequality} holds for all $f\in V_{\varphi}^{\star}(R,S,\delta)$ with probability at the least $1-\epsilon$ if 
\begin{align*}
3a\exp\Bigg(-b\frac{\mu^2 lm(1-\delta)^{2pq}D^{2(1-pq)}}{12R^{2np}S^{2q}+ D^{1-pq}R^{np}S^q}\Bigg)< \epsilon.
\end{align*}
For small $\delta$, $b=\frac{2^{\frac{1}{4\log2}-6}(\log2)^{4}}{(n+3)^4}> \frac{1}{2^{10}(n+3)^4}.$ Hence
\begin{align*}
\exp\Bigg(G(R+S)^{n^2+n}-\frac{1}{2^{10}(n+3)^4}\times \frac{\mu^2 lm(1-\delta)^{2pq} D^{2(1-pq)}}{12R^{2np}S^{2q}+ D^{1-pq}R^{np}S^q} \Bigg)<\frac{\epsilon}{3}.
\end{align*}
Thus
\begin{align*}
lm> 2^{10}(n+3)^4\Big(G(R+S)^{n^2+n}+ \log\Big(\frac{3}{\epsilon}\Big)\Big)\frac{12R^{2np}S^{2q}+ D^{1-pq}R^{np}S^q}{\mu^2 lm(1-\delta)^{2pq} D^{2(1-pq)}}.
\end{align*}
Thereby, if the random sample set satisfy above inequality, then the sample set is a stable set of sampling for the set $V_{\varphi}^{\star}(R,S,\delta)$ with high probability.
\end{eg}

In the following, we propose a reconstruction algorithm for the set of concentration function. Taking note from \cite{LSX}, we use the iteration scheme to reconstruct functions in $V^{\star}(R,S,\delta)$ from their random sample values.
\begin{theorem}
If $X=\{(x_{i},y_{j}): 1 \leq i \leq l,1 \leq j \leq m\}$ is a random sample set satisfying the sampling inequality \eqref{samplinginequality} and $\theta$ is a gap of the set $X$. Then for any $f \in V^{\star}(R,S,\delta)$ can be uniquely reconstructed from its sample $\{f(x_i,y_j)\}$ with the help of following iterative algorithm:
$$f_0= S_X f, \hspace{1cm} f_{r}=f_0+f_{r-1}-S_X f_{r-1}, \hspace{5mm} r \geq 1,$$
where $\displaystyle S_X f= \sum_{j=1}^m\sum_{i=1}^l f(x_i,y_j)T\beta_{i,j}$ and $\beta_{i,j}$ is the partition of unity supported on $B_{\theta}(x_i,y_j)$. Moreover, $f_r$ converges to $f$ exponentially, and
\begin{align*}
\|f_r-f\|_{L^{p,q}(\RR^{n+1})}\leq \frac{1 +\|K\|_{W}(\|w_{\theta}(K)\|_{W}+\delta)}{1 -\|K\|_{W}(\|w_{\theta}(K)\|_{W}+\delta)} \big[\|K\|_{W}(\|w_{\theta}(K)\|_{W}+\delta)\big]^{r+1} \|f\|_{L^{p,q}(\RR^{n+1})}.
\end{align*}
\end{theorem}

\begin{proof}
From the regularity condition of the kernel $K$, $\lim\limits_{\epsilon\to 0}  \|w_{\epsilon}(K)\|_{W} = 0$. Therefore, for $\frac{1}{\|K\|_{W}}-\delta>0$ there exists $\epsilon'>0$ such that
$$\|w_{\epsilon}(K)\|_W<\frac{1}{\|K\|_{W}}-\delta,\hspace{6mm}\text{whenever } 0 <\epsilon<\epsilon'.$$ We consider $\eta<\min\{\frac{2}{n},\epsilon'\}$. $X$ is a finite collection of sample set in $C_{R,S}$ with gap $\theta\big(<\eta \text{ as } lm>\oo\big(N_0(\Gamma)(R+S)^{n+1}\big)\big).$ The collection $\{B_{\theta}(x_i,y_j): 1\leq i\leq l, 1\leq j\leq m\}$ is a cover for $C_{R,S}$ and $\beta_{i,j}$ is the partition of unity with $supp (\beta_{i,j}) \subseteq B_{\theta}(x_i,y_j)$ and satisfies
\begin{enumerate}
\item $0 \leq \beta_{i,j} \leq 1.$
\item $\sum\limits_{j=1}^m \sum\limits_{i=1}^l \beta_{i,j} \equiv 1.$
\end{enumerate}
Now consider the iterative algorithm
$$f_0=S_X f, \hspace{1cm} f_r=f_0+f_{r-1}-S_X f_{r-1}, \hspace{5mm} r \geq 1.$$
For every $(x,y)\in \RR^{n+1}$, we consider $\displaystyle Q_{X}f(x,y)=\sum_{j=1}^m \sum_{i=1}^l f(x_i,y_j)\beta_{i,j}(x,y),$ then $S_X=TQ_X$. Now, for $f \in V^{\star}(R,S,\delta),$ we have
\begin{align*}
\|f-S_X f\|_{L^{p.q}(\RR^{n+1})}&= \|Tf-TQ_{f}f\|_{L^{p.q}(\RR^{n+1})} \\
&\leq \|K\|_{W}(\|f-Q_{f}f\|_{L^{p.q}(C_{R,S})}+\|f\|_{L^{p.q}(C_{R,S}^c)}). 
\end{align*}
For $(x,y)\in C_{R,S}$
\begin{align*}
|f(x,y)-Q_{X}f(x,y)| \leq \sum_{j=1}^m \sum_{i=1}^l |f(x,y)-f(x_i,y_j)|\beta_{i,j}(x,y) \leq w_{\theta}(f)(x,y).
\end{align*}
Since $\displaystyle w_{\theta}(f)(x,y) \leq \int_{\RR}\int_{\RR^{n}} w_{\theta}(K)(x,y,s,t)f(s,t)\,dsdt,$ we get
\begin{align*}
\|w_{\theta}(f)\|_{L^{p,q}(\RR^{n+1})}\leq \|w_{\theta}(K)\|_W \|f\|_{L^{p,q}(\RR^{n+1})}.
\end{align*}
Therefore,
\begin{align*}
\|f-S_X f\|_{L^{p.q}(\RR^{n+1})}\leq \|K\|_W\big(\|w_{\theta}(K)\|_W+\delta\big) \|f\|_{L^{p,q}(\RR^{n+1})}.
\end{align*}
From the definition of iteration algorithm, we have 
$$f_r -f_{r-1}=(I-S_X)f_0,~~\text{and } f_r= S_X f+ \sum_{h=1}^{r}(T-S_X)^h S_X f,$$
where $I$ is an identity operator. Now, we define $\displaystyle R=I+\sum_{h=1}^{\infty}(T-S_X)^h,$ then $R$ is a bounded operator and
\begin{align*}
\|Rf\|_{L^{p,q}(\RR^{n+1})}&\leq\sum_{h=1}^{\infty}\|(I-S_X)^h Tf\|_{L^{p,q}(\RR^{n+1})} \\
&\leq \frac{\|K\|_{W}}{1-\|K\|_W \big( \|w_{\theta}(K)\|_W+\delta \big)} \|f\|_{L^{p,q}(\RR^{n+1})}. 
\end{align*}
Therefore, $R$ is psedo-inverse of pre-contraction operator $S_{X},$ i.e., $RS_{X}f = S_{X}Rf = f.$ This implies,
\begin{align*}
\|f_r- f\|_{L^{p,q}(\RR^{n+1})} &= \|\sum_{h=r+1}^{\infty}(I-S_{X})^{h}S_{X}f\|_{L^{p,q}(\RR^{n+1})}  \\
&\leq \sum_{h=r+1}^{\infty}\|I-S_X\|^h\|S_X f\|_{L^{p,q}(\RR^{n+1})}  \\
&\leq \frac{1+\|K\|_W(\|w_{\theta}(K)\|_W+\delta)}{1-\|K\|_{W}(\|w_{\theta}(K)\|_W+ \delta)}\big[\|K\|_W(\|w_{\theta}(K)\|_W+\delta)\big]^{r+1}\|f\|_{L^{p,q}(\RR^{n+1})}.
\end{align*}
As $\|K\|_W(\|w_{\theta}(K)\|_W+\delta)<1,$ therefore, $\|f_r-f\|_{L^{p,q}(\RR^{n+1})}\to 0$ exponentially as $r\to \infty.$
This complete the proof.
\end{proof}

\section*{Acknowledgement}
P. Goyal and S. Sivananthan acknowledges the financial support through project no. CRG/2019/002412 funded by the Department of Science and Technology, Government of India. D. Patel acknowledges Council of Scientic and Industrial Research for the financial support. 

\end{document}